\documentclass[11pt]{article}
\usepackage{amsfonts, amssymb, amsmath, amsthm, xcolor, graphicx}
\usepackage[colorlinks=true,citecolor=blue,linkcolor=blue, urlcolor=blue]{hyperref}

\begin{document}

\theoremstyle{plain}
\newtheorem{thm}{Theorem}[section]
\newtheorem{cor}[thm]{Corollary}
\newtheorem{con}[thm]{Conjecture}
\newtheorem{cla}[thm]{Claim}
\newtheorem{lm}[thm]{Lemma}
\newtheorem{prop}[thm]{Proposition}
\newtheorem{example}[thm]{Example}

\theoremstyle{definition}
\newtheorem{dfn}[thm]{Definition}
\newtheorem{alg}[thm]{Algorithm}
\newtheorem{prob}[thm]{Problem}
\newtheorem{rem}[thm]{Remark}

\newcommand{\A}{\mathcal{A}}
\newcommand{\B}{\mathcal{B}}
\newcommand{\D}{\mathcal{D}}
\newcommand{\E}{\mathcal{E}}
\newcommand{\F}{\mathcal{F}}

\renewcommand{\baselinestretch}{1.1}

\title{An Erd{\H o}s-Ko-Rado theorem for permutations with fixed number of cycles}
\author{
Cheng Yeaw Ku
\thanks{ Department of Mathematics, National University of
Singapore, Singapore 117543. E-mail: matkcy@nus.edu.sg} \and Kok
Bin Wong \thanks{
Institute of Mathematical Sciences, University of Malaya, 50603
Kuala Lumpur, Malaysia. E-mail:
kbwong@um.edu.my.} } \maketitle

\begin{abstract}\noindent
Let $S_{n}$ denote the set of permutations of $[n]=\{1,2,\dots, n\}$. For a positive integer $k$, define $S_{n,k}$ to be the set of all permutations of $[n]$ with exactly $k$ disjoint cycles, i.e.,
\[ S_{n,k} = \{\pi \in S_{n}: \pi = c_{1}c_{2} \cdots c_{k}\},\] 
where $c_1,c_2,\dots ,c_k$ are disjoint cycles. The size of $S_{n,k}$ is given by $\left [ \begin{matrix}n\\ k \end{matrix}\right]=(-1)^{n-k}s(n,k)$, where $s(n,k)$ is the  Stirling number of the first kind. A family $\mathcal{A} \subseteq S_{n,k}$ is said to be $t$-{\em intersecting} if any two elements of $\mathcal{A}$ have at least $t$ common cycles. In this paper, we show that, given any positive integers $k,t$ with $k\geq t+1$, there exists an integer $n_0=n_0(k,t)$, such that for all $n\geq n_0$, if $\mathcal{A} \subseteq S_{n,k}$ is $t$-intersecting, then 
\[ |\mathcal{A}| \le \left [ \begin{matrix}n-t\\ k-t \end{matrix}\right],\]
with equality if and only if $\mathcal{A}$ is the stabiliser of $t$ fixed points.

\end{abstract}

\bigskip\noindent
{\sc keywords:}  $t$-intersecting family, Erd{\H o}s-Ko-Rado, permutations

\section{Introduction}

Let $[n]=\{1, \ldots, n\}$, and let ${[n] \choose k}$ denote the
family of all $k$-subsets of $[n]$. A family $\mathcal{A}$ of subsets of $[n]$ is $t$-{\em intersecting} if $|A \cap B| \ge t$ for all $A, B \in \mathcal{A}$. One of the most beautiful results in extremal combinatorics is
the Erd{\H o}s-Ko-Rado theorem.

\begin{thm}[Erd{\H o}s, Ko, and Rado \cite{EKR}, Frankl \cite{Frankl}, Wilson \cite{Wilson}]\label{EKR} Suppose $\mathcal{A} \subseteq {[n] \choose k}$ is $t$-intersecting and $n>2k-t$. Then for $n\geq (k-t+1)(t+1)$, we have
\begin{equation}
\vert \mathcal{A} \vert\leq {n-t \choose k-t}.\notag
\end{equation}
Moreover, if $n>(k-t+1)(t+1)$ then equality holds if and only if $\mathcal{A}=\{A\in {[n] \choose k}\ :\ T\subseteq A\}$ for some $t$-set $T$.
\end{thm}

Later, Ahlswede and Khachatrian \cite{AK} extended the Erd{\H o}s-Ko-Rado theorem by determining the structure of all  $t$-intersecting set systems of maximum size for all possible $n$ (see also \cite{Bey, FT, Kee, Ku_Wong3, MT, Pyber, Toku} for some related results). There have been many recent results showing that a version of the Erd{\H o}s-Ko-Rado theorem holds for combinatorial objects other than set systems (see \cite{AK2,B, B2, B3, BH,CK, CP, DF, E, EFP, FW, GM, HS, HST, HT, HK, KL, KR,KW, Ku_Wong, Ku_Wong2, Ku_Wong4,LM, LW, Moon, WZ, W}). Most notably is the results of Ellis, Friedgut and Pilpel \cite{EFP} who showed that for sufficiently large $n$ depending on $t$, a $t$-intersecting family $\A$ of permutations has size at most $(n-t)!$, with equality if and only if $\A$ is a coset of the stabilizer of $t$ points, thus settling an old conjecture of Deza and Frankl in the affirmative. The proof uses spectral methods and representations of the symmetric group. 

Let $S_{n}$ denote the set of permutations of $[n]$. For a positive integer $k$, define $S_{n,k}$ to be the set of all permutations of $[n]$ with exactly $k$ disjoint cycles, i.e.,
\[ S_{n,k} = \{\pi \in S_{n}: \pi = c_{1}c_{2} \cdots c_{k}\},\] 
where $c_1,c_2,\dots ,c_k$ are disjoint cycles. It is well known that the size of $S_{n,k}$ is given by $\left [ \begin{matrix}n\\ k \end{matrix}\right]=(-1)^{n-k}s(n,k)$, where $s(n,k)$ is the {\em Stirling number of the first kind}.

We shall use the following notations:
\begin{itemize}
\item[\textnormal{(a)}] $N(c)=\{a_1,a_2,\dots, a_l\}$ for a cycle $c=(a_1,a_2,\dots, a_l)$;
\item[\textnormal{(b)}] $M(\pi)=\{c_1,c_2,\dots, c_k\}$ for a $\pi=c_1c_2\dots c_k\in S_{n,k}$; 
\end{itemize}

A family $\mathcal{A} \subseteq S_{n,k}$ is said to be $t$-{\em intersecting} if any two elements of $\mathcal{A}$ have at least $t$ common cycles, i.e., $\vert M(\pi_1)\cap M(\pi_2)\vert\geq t$ for all $\pi_1,\pi_2\in \A$.

\begin{thm}\label{thm_main} Given any positive integers $k,t$ with $k\geq t+1$, there exists an integer $n_0=n_0(k,t)$, such that for all $n\geq n_0$, if $\mathcal{A} \subseteq S_{n,k}$ is $t$-intersecting, then 
\[ |\mathcal{A}| \le \left [ \begin{matrix}n-t\\ k-t \end{matrix}\right],\]
with equality if and only if $\mathcal{A}$ is the stabiliser of $t$ fixed points.
\end{thm}

\section{Stirling number revisited}

The unsigned Stirling number $\left [ \begin{matrix}n\\ k \end{matrix}\right]$ satisfies the recurrence relation
\begin{equation}\label{eq_recurrence}
\left [ \begin{matrix}n\\ k \end{matrix}\right]=\left [ \begin{matrix}n-1\\ k-1 \end{matrix}\right]+(n-1)\left [ \begin{matrix}n-1\\ k \end{matrix}\right],
\end{equation}
with initial conditions $\left [ \begin{matrix}0\\ 0 \end{matrix}\right] = 1$ and $\left [ \begin{matrix}n\\ 0 \end{matrix}\right] = \left [ \begin{matrix} 0\\ k \end{matrix}\right] = 0$,  $n > 0$. Note that $\left [ \begin{matrix}n\\ n\end{matrix}\right]=1$.
By using equation (\ref{eq_recurrence}) and induction on $n$, 
\begin{equation}\label{eq_base_case}
\left [ \begin{matrix}n\\ 1\end{matrix}\right]=(n-1)!.
\end{equation}
By applying  equation (\ref{eq_recurrence}) repeatedly, 
\begin{align}\label{eq_full_series}
\left [ \begin{matrix}n\\ k \end{matrix}\right] &=\left [ \begin{matrix}n-1\\ k-1 \end{matrix}\right]+(n-1)\left [ \begin{matrix}n-1\\ k \end{matrix}\right]\notag\\
&=\left [ \begin{matrix}n-1\\ k-1 \end{matrix}\right]+(n-1)\left [ \begin{matrix}n-2\\ k-1 \end{matrix}\right]+(n-1)(n-2)\left [ \begin{matrix}n-2\\ k \end{matrix}\right]\notag\\
&=\left [ \begin{matrix}n-1\\ k-1 \end{matrix}\right]+(n-1)\left [ \begin{matrix}n-2\\ k-1 \end{matrix}\right]+(n-1)(n-2)\left [ \begin{matrix}n-3\\ k-1 \end{matrix}\right]+(n-1)(n-2)(n-3)\left [ \begin{matrix}n-3\\ k \end{matrix}\right]\notag\\
&\vdots\notag\\
&=\sum_{r=k-1}^{n-1} \frac{(n-1)!}{r!}\left [ \begin{matrix} r\\ k-1 \end{matrix}\right].
\end{align}
In particular (by  equations (\ref{eq_base_case}) and (\ref{eq_full_series})),
\begin{equation}\label{eq_for_k2}
\left [ \begin{matrix}n\\ 2 \end{matrix}\right]=(n-1)! \sum_{r=1}^{n-1} \frac{1}{r}.
\end{equation}
By elementary calculus, it is easy to show that  for sufficiently large $n$, 
\begin{equation}\label{eq_calculus1}
\frac{\ln n}{2}<\ln (n-1)+\frac{1}{n-1}\leq \sum_{r=1}^{n-1} \frac{1}{r}\leq \ln (n-1)+1<2\ln n.
\end{equation}
Hence, there  are positive constants $\alpha(2)$ and $\beta(2)$ such that 
\begin{equation}\label{eq_ineqaulity_k2}
\alpha(2)((n-1)!(\ln n))<\left [ \begin{matrix}n\\ 2 \end{matrix}\right]<\beta(2)((n-1)!(\ln n)),
\end{equation}
for all $n\geq 2$.

Again, by elementary calculus, for $m\geq 1$ and sufficiently large $n$ depending on $m$, 
\begin{equation}\label{eq_calculus2}
\frac{\ln^{m+1} n}{2(m+1)}<\sum_{r=1}^{n-1} \frac{\ln^m r}{r}<\frac{2 \ln^{m+1} n}{m+1}.
\end{equation}
The following lemma follows by using  equations (\ref{eq_full_series}) and (\ref{eq_calculus2}), and induction on $k$.
\begin{lm}\label{lemma_basic_inequality} There are positive constants $\alpha(k)$ and $\beta(k)$ such that
\begin{equation}
\alpha(k)((n-1)!(\ln^{k-1} n))<\left [ \begin{matrix}n\\ k \end{matrix}\right]<\beta(k)((n-1)!(\ln^{k-1} n)),\notag
\end{equation}
for all $n\geq k$. 
\end{lm} 

\section{Main results}

A family $\B\subseteq S_{n,k}$ is said to be \emph{independent} if $M(\pi_1)\cap M(\pi_2)=\varnothing$ for all $\pi_1,\pi_2\in\B$. 

\begin{lm}\label{lemma_maximum_size_independent} Let $\A\subseteq S_{n,k}$ and $k\geq 2$. If a maximal independent subset of $\A$ is of size at most $l$, then
\begin{equation}
\vert\A\vert\leq kl\left [ \begin{matrix}n-1\\ k-1 \end{matrix}\right].\notag
\end{equation}
\end{lm}

\begin{proof} Let $\B=\{\pi_1,\pi_2,\dots ,\pi_l\}$ be a maximal independent subset of $\A$ of size $l$. Let $\pi_i=c_{i1}c_{i2}\dots c_{ik}$ where $c_{i1},c_{i2},\dots ,c_{ik}$ are disjoint cycles, and 
\begin{equation}
Q=\bigcup_{i=1}^l M(\pi_i).\notag
\end{equation} 
Note that $\vert Q\vert =kl$. 

Let 
\begin{equation}
\A_{ij}=\{\pi \in \A\ :\ c_{ij}\in M(\pi)\}.\notag
\end{equation}
Let $\pi\in\A\setminus \B$. By the maximality of $\B$, $M(\pi)\cap Q\neq \varnothing$. So, $c_{ij}\in M(\pi)$ for some $i,j$, and $\pi\in \A_{ij}$. Hence,
\begin{equation}
\A=\bigcup_{i,j} \A_{ij}.\notag
\end{equation}
The lemma follows by noting that
\begin{equation}
\vert A_{ij}\vert\leq \left [ \begin{matrix}n-\vert N(c_{ij})\vert\\ k-1 \end{matrix}\right]\leq \left [ \begin{matrix}n-1\\ k-1 \end{matrix}\right].\notag
\end{equation}
\end{proof}

Let $\A\subseteq S_{n,k}$ and $c_1,\dots ,c_t$ by cycles such that $N(c_i)\cap N(c_j)=\varnothing$ for $i\neq j$. Let $T=\{c_1,\dots, c_t\}$. We set
\begin{equation}
\A(T)=\{\pi\in\A\ :\ T\subseteq M(\pi)\}.\notag
\end{equation}
Now, for each element $\pi\in\A(T)$, we remove all the cycles $c_1,c_2,\dots, c_t$ from $\pi$ and denote the resulting set by $\A^*(T)$. Let $P=\bigcup_{i=1}^t N(c_i)$. Note that $\vert\A(T)\vert=\vert\A^*(T)\vert$ and $\A^*(T)\subseteq S_{n-\vert P\vert,k-t}$. Here, $S_{n-\vert P\vert,k-t}$ is the set of all permutations of $[n]\setminus P$ with exactly $k-t$ disjoint cycles.

\begin{lm}\label{lm_main_independent} Let $\A\subseteq S_{n,k}$ be maximal $t$-intersecting and $k\geq t+1$. Let $T=\{c_1,\dots, c_t\}$ with $N(c_i)\cap N(c_j)=\varnothing$ for $i\neq j$. If $\mathcal A^*(T)$ has an independent set of size at least $k+1$, then
\begin{equation}
\A=\left \{ \pi \in S_{n,k}\ :\  T\subseteq M(\pi)\right\}.\notag
\end{equation}
\end{lm}

\begin{proof} Let $\{\pi_1,\dots, \pi_{k+1}\}$ be an independent subset of $\mathcal A^*(T)$ of size $k+1$. For $l=1,2,\dots, k+1$, let 
\begin{equation}
\pi_l= c_1\dots c_td_{l,t+1}\dots d_{l,k},\notag
\end{equation}
where $c_1,\dots,c_t,d_{l,t+1},\dots, d_{l,k}$ are disjoint cycles. 

Suppose there is a $\pi\in\A$ such that $c_{i_0}\notin M(\pi)$ for a fixed $i_0$. Since $\A$ is $t$-intersecting,
\begin{equation}
\vert M(\pi)\cap M(\pi_l)\vert\geq t.\notag
\end{equation}
Therefore there is a $j_l$ ($t+1\leq j_l\leq k$) with $d_{l,j_l}\in M(\pi)$. Since all the $d_{l,j_l}$ are distinct, $k=\vert M(\pi)\vert\geq k+1$, a contradiction. Hence, 
\begin{equation}
\A\subseteq \left \{ \pi \in S_{n,k}\ :\  T\subseteq M(\pi)\right\}.\notag
\end{equation}
By the maximality of $\mathcal A$, the lemma follows.
\end{proof}

\begin{proof}[Proof of Theorem \ref{thm_main}] We may assume that $\A$ is maximal $t$-intersecting.

Suppose $k=t+1$. Since $\A$ is $t$-intersecting, there are $\pi_1,\pi_2\in\A$ such that
\begin{align}
\pi_1&=c_1c_2\dots c_td_1\notag\\
\pi_2&=c_1c_2\dots c_td_2\notag
\end{align}
where $c_1,\dots,c_t,d_1$ are disjoint cycles, $d_2\neq d_1$ and $N(d_2)=N(d_1)$. Suppose there is a $\pi\in\A$ with $c_{i_0}\notin\ M(\pi)$ for some $i_0$. Then $d_1,d_2\in M(\pi)$. But this is impossible as $N(d_1)=N(d_2)$. Hence,
\begin{equation}
\A= \{ \pi\in S_{n,k}\ :\ c_i\in M(\pi)\ \ \textnormal{for}\ \ i=1,2,\dots, t\},\notag
\end{equation}
for $\A$ is maximal $t$-intersecting. Let  $P=\bigcup_{i=1}^t N(c_i)$. Then
\begin{equation}
\vert\A\vert=\left [ \begin{matrix}n-\vert P\vert\\ 1 \end{matrix}\right]\leq \left [ \begin{matrix}n-t\\ 1 \end{matrix}\right],\notag
\end{equation}
with equality if and only if $\vert N(c_i)\vert=1$ for $i=1,2\dots,t$, i.e., $\A$ is the stabilizer of at least $t$ fixed points. Now, if $n\geq t+2$, then  $\A$ is the stabilizer of $t$ fixed points.

Suppose $k\geq t+2$.  Let $\pi_0=d_1d_2\dots d_k\in\A$ be fixed, where $d_1,\dots, d_k$ are disjoint cycles. Then
\begin{equation}
\A=\bigcup_{T\subseteq M(\pi_0),\vert T\vert=t} \A(T).\notag
\end{equation}

\vskip 0.5cm
\noindent
{\bf Case 1}. Suppose that for each $T\subseteq M(\pi_0)$ with $\vert T\vert=t$, all independent subsets of $\A^*(T)$ is of size at most $k$. Then by Lemma \ref{lemma_maximum_size_independent} and equation (\ref{eq_recurrence}),
\begin{equation}
\vert\A(T)\vert=\vert\A^*\vert\leq k^2\left [ \begin{matrix}n-\vert P\vert-1\\ k-t-1 \end{matrix}\right]\leq k^2\left [ \begin{matrix}n-t-1\\ k-t-1 \end{matrix}\right],\notag
\end{equation}
where $P=\bigcup_{c\in T} N(c)$. This implies that
\begin{equation}
\vert\A\vert\leq k^2\binom{k}{t}\left [ \begin{matrix}n-t-1\\ k-t-1 \end{matrix}\right].\notag
\end{equation}
By Lemma \ref{lemma_basic_inequality}, there is are positive constants $\alpha$ and $\beta$ such that
\begin{equation}
\vert\A\vert<\beta k^2\binom{k}{t}\left((n-t-2)!(\ln^{k-t-2} n)\right),\notag
\end{equation}
and
\begin{equation}
\alpha\left((n-t-2)!(\ln^{k-t-1} n)\right)<\left [ \begin{matrix}n-t-1\\ k-t \end{matrix}\right].\notag
\end{equation}
So, for sufficiently large $n$, $\vert\A\vert<\left [ \begin{matrix}n-t-1\\ k-t \end{matrix}\right]$, and by equation (\ref{eq_recurrence}), $\vert \A\vert<\left [ \begin{matrix}n-t\\ k-t \end{matrix}\right]$.

\vskip 0.5cm
\noindent
{\bf Case 2}. Suppose that there is a $T_0\subseteq M(\pi_0)$ with $\vert T_0\vert=t$, such that $\A^*(T_0)$ has an independent set of size at least $k+1$. By Lemma \ref{lm_main_independent}, 
\begin{equation}
\A=\left \{ \pi \in S_{n,k}\ :\  T_0\subseteq M(\pi)\right\}.\notag
\end{equation}
Let $P_0=\bigcup_{c\in T_0} N(c_0)$. Then
\begin{equation}
\vert \A\vert=\left[ \begin{matrix}n-\vert P_0\vert\\ k-t \end{matrix}\right]\leq \left[ \begin{matrix}n-t\\ k-t \end{matrix}\right],\notag
\end{equation}
with equality if and only if $\vert N(c_i)\vert=1$ for $i=1,2\dots,t$, i.e., $\A$ is the stabilizer of at least $t$ fixed points. 
\end{proof}

\section*{Acknowledgments}	

This project is supported by the Advanced Fundamental Research Cluster, University of Malaya (UMRG RG238/12AFR).

\end{document}